\documentclass[11pt,a4paper]{article}

\usepackage[a4paper,margin=28mm]{geometry}
\usepackage{amsmath,amssymb,amsthm}
\usepackage{xcolor}
\usepackage{url}
\usepackage{hyperref}
\usepackage{microtype}

\hypersetup{
  colorlinks=true,
  linkcolor=blue,
  citecolor=blue,
  urlcolor=blue
}

\newtheorem{thm}{Theorem}[section]

\newcommand{\R}{\mathbb{R}}

\newcommand{\rmd}{\mathrm{d}}

\renewcommand{\bar}{\overline}

\title{Modified ruin probability for a Cram{\'e}r-Lundberg model driven by a compound mixed Poisson process}

\author{Noriyoshi Sakuma\thanks{Department of Mathematics, Graduate School of Science, The University of Osaka, 1 Machikaneyama-cho, Toyonaka, Osaka 560-0043, JAPAN; Graduate School of Science, Nagoya City University, 1 Yamanohata, Mizuho-cho, Mizuho-ku, Nagoya, Aichi 467-8501, Japan. E-mail: \texttt{sakuma.noriyoshi.sci@osaka-u.ac.jp}}\and Momoka Tashiro\thanks{Graduate School of Science, Nagoya City University, 1 Yamanohata, Mizuho-cho, Mizuho-ku, Nagoya, Aichi 467-8501, Japan. E-mail: \texttt{}}}
\date{}

\newcommand{\keywords}[1]{\par\medskip\noindent\textbf{Keywords.} #1\par\medskip}
\newcommand{\references}{}

\begin{document}
\maketitle

\begin{abstract}
We study modified ruin probabilities in a Cramér-Lundberg model driven by a compound mixed Poisson process. In the heavy-tailed regime, if the integrated claim-size distribution is subexponential and the upper endpoint of the mixing distribution stays below the net-profit boundary, the modified and classical ruin probabilities are asymptotically equivalent. In the light-tailed regime, we prove a fixed-intensity ratio theorem and obtain both an endpoint-atom result and a sharp endpoint-density asymptotic with an explicit constant.
\end{abstract}


\section{Introduction}

Modified ruin probabilities extend classical ruin by allowing the surplus process to survive a negative excursion according to a prescribed rule; see \cite{AB}. This framework contains, for example, Parisian and cumulative Parisian ruin as special cases; see \cite{DassiosWu08,GR}. These results are usually formulated for a fixed arrival mechanism. Separately, mixed Poisson claim arrivals model heterogeneity in the claim intensity; see \cite{Grandell,Mikosch}. For classical ruin probabilities, such heterogeneity changes the averaged prefactor in the heavy-tailed case and, in the light-tailed case, makes the largest admissible intensities near the endpoint of the mixing distribution asymptotically dominant. The combination is not covered by the two theories separately, because the modified ruin rule depends on the overshoot at ruin, whereas the mixture selects different adjustment coefficients through the random intensity.

This note identifies that interaction in both tail regimes. In the heavy-tailed case we prove that the modified and classical mixed ruin probabilities are asymptotically equivalent. In the light-tailed Cram{\'e}r case we first isolate the fixed-intensity overshoot mechanism, then show that an atom at the upper endpoint of the mixing distribution gives the leading term, and finally obtain a sharp explicit constant under a regular endpoint-density assumption. Thus the light-tailed results clarify how intensity heterogeneity and modified-ruin survival jointly determine the leading asymptotics. Endpoints without either an atom or a regular density remain an interesting separate problem.

\section{Model}

Let
\[
U_t=u+ct-\sum_{i=1}^{N_t}X_i,\qquad t\ge0,
\]
where $u>0$ and $c>0$. The claim sizes $\{X_i\}_{i\ge1}$ are i.i.d.\ positive random variables with distribution function $F$, tail $\bar F=1-F$, and mean $\mu=\mathbb{E}[X_1]$. The claim number process $\{N_t\}_{t\ge0}$ is a mixed Poisson process with mixing variable $\Lambda$ and distribution function $G$, independent of the claim sizes. Write
\[
\ell_1:=\inf\{\ell\ge0:G(\ell)=1\}.
\]

Let
\[
T:=\inf\{t>0:U_t<0\}
\]
be the classical ruin time. For a fixed intensity $\ell$, set
\[
\psi_{\rm cl}(u,\ell):=\mathbb{P}_u(T<\infty\,|\, \Lambda=\ell).
\]
The mixed-Poisson classical ruin probability is
\[
\psi_{\rm cl}(u):=\int_{[0,\infty)}\psi_{\rm cl}(u,\ell)\,G(\rmd \ell),
\]
with the convention $\psi_{\rm cl}(u,\ell)=1$ when $\ell\mu\ge c$. Following \cite{AB}, a modified ruin probability is a measurable function $\psi:\R\to[0,1]$ such that
\begin{equation}\label{eq:modified}
\psi(u)=\int_{(-\infty,0)}\psi(y)\,\mathbb{P}_u(U_T\in \rmd y,\,T<\infty),\qquad u\ge0.
\end{equation}
For later use, define also
\[
\psi_\ell(u):=\int_{(-\infty,0)}\psi(y)\,\mathbb{P}_u(U_T\in \rmd y,\,T<\infty\,|\, \Lambda=\ell),
\]
so that $\psi(u)=\int_{[0,\infty)}\psi_\ell(u)\,G(\rmd\ell)$ and $0\le \psi_\ell(u)\le \psi_{\rm cl}(u,\ell)$.

\section{Heavy-tailed asymptotics}

Define the integrated tail distribution by
\[
F_I(x):=\frac{1}{\mu}\int_0^x \bar F(y)\,\rmd y,\qquad x\ge0.
\]

\begin{thm}\label{thm:heavy}
Assume that $\ell_1<c/\mu$, that $F_I\in\mathcal S$, and that
\[
\lim_{y\to-\infty}\psi(y)=1.
\]
Then
\[
\psi(u)\sim\psi_{\rm cl}(u)\sim\mathbb{E}\!\left[\frac{\Lambda\mu}{c-\Lambda\mu}\right]\bar F_I(u),\qquad u\to\infty.
\]
\end{thm}

\begin{proof}
Fix $\ell\in(0,\ell_1]$. The standard subexponential ruin asymptotic for the fixed-rate Cramér-Lundberg model gives
\[
\psi_{\rm cl}(u,\ell)\sim a(\ell)\bar F_I(u),\qquad a(\ell):=\frac{\ell\mu}{c-\ell\mu};
\]
see \cite{Mikosch,AsmussenAlbrecher10}. Since every subexponential distribution is long-tailed,
\[
\frac{\psi_{\rm cl}(u+z,\ell)}{\psi_{\rm cl}(u,\ell)}\to1,\qquad z\in\R.
\]
Hence Theorem~2 of Schmidli \cite{Schmidli} yields, for every $x\ge0$,
\[
\mathbb{P}_u(-U_T>x\,|\, T<\infty,\Lambda=\ell)\to1.
\]
Therefore
\[
\inf_{y\le -x}\psi(y)\,\mathbb{P}_u(-U_T>x\,|\, T<\infty,\Lambda=\ell)
\le \frac{\psi_\ell(u)}{\psi_{\rm cl}(u,\ell)}\le1,
\]
and letting first $u\to\infty$ and then $x\to\infty$ gives $\psi_\ell(u)\sim\psi_{\rm cl}(u,\ell)$.

Next, $\psi_{\rm cl}(u,\ell)$ is increasing in $\ell$, so $\psi_{\rm cl}(u,\ell)\le \psi_{\rm cl}(u,\ell_1)$. Since $\psi_{\rm cl}(u,\ell_1)\sim a(\ell_1)\bar F_I(u)$, we have for all large $u$,
\[
0\le \frac{\psi_\ell(u)}{\bar F_I(u)}\le \frac{\psi_{\rm cl}(u,\ell_1)}{\bar F_I(u)}\le 2a(\ell_1).
\]
Dominated convergence therefore yields
\begin{align*}
\psi(u)&=\int_{[0,\infty)}\psi_\ell(u)\,G(\rmd\ell)\\
&\sim \int_{[0,\infty)}a(\ell)\,G(\rmd\ell)\,\bar F_I(u)
=\mathbb{E}\!\left[\frac{\Lambda\mu}{c-\Lambda\mu}\right]\bar F_I(u).
\end{align*}
The same argument with $\psi_\ell(u)$ replaced by $\psi_{\rm cl}(u,\ell)$ gives the same asymptotic for $\psi_{\rm cl}(u)$, and the claim follows.
\end{proof}

\section{Light-tailed asymptotics}

\begin{thm}\label{thm:light-fixed}
Fix $\ell\in(0,c/\mu)$. Assume that
\[
\psi_{\rm cl}(u,\ell)\sim c_\ell e^{-R_\ell u},\qquad u\to\infty,
\]
for some constants $c_\ell,R_\ell>0$. If $\psi$ is continuous or monotone on $(-\infty,0)$, then
\[
\psi_\ell(u)\sim C_\ell\,\psi_{\rm cl}(u,\ell),\qquad u\to\infty,
\]
where
\[
C_\ell:=\int_{(-\infty,0)}\psi(y)\,\nu_\ell(\rmd y)
\]
and the probability measure $\nu_\ell$ is determined by
\begin{align*}
&\nu_\ell(( -\infty,-x])
\\
&=\frac{1}{c-\ell\mu}\Bigl(ce^{-R_\ell x}
\\
&\qquad\qquad-\ell\int_0^x e^{-R_\ell(x-z)}\bar F(z)\,\rmd z
\\
&\qquad\qquad-\ell\int_x^{\infty}\bar F(z)\,\rmd z\Bigr),
\end{align*}
for $x\ge0$.
\end{thm}

\begin{proof}
From the assumed asymptotic,
\[
\frac{\psi_{\rm cl}(u+z,\ell)}{\psi_{\rm cl}(u,\ell)}\to e^{-R_\ell z},\qquad z\in\R.
\]
Theorem~2 of Schmidli \cite{Schmidli} therefore implies that $\mathcal L(U_T\,|\, T<\infty,\Lambda=\ell)$ converges weakly to a probability measure $\nu_\ell$ whose tail is given above. The displayed formula is continuous in $x$, so $\nu_\ell$ has no atoms. Since
\[
\frac{\psi_\ell(u)}{\psi_{\rm cl}(u,\ell)}=\int_{(-\infty,0)}\psi(y)\,\mathbb{P}_u(U_T\in\rmd y\,|\, T<\infty,\Lambda=\ell),
\]
the weak convergence yields the stated asymptotic when $\psi$ is continuous. The monotone case follows as well because the limit law is continuous.
\end{proof}

\begin{thm}[Endpoint atom case]\label{thm:light-mixed}
Assume that $\ell_1<c/\mu$, that $p_1:=G(\{\ell_1\})>0$, and that the assumptions of Theorem~\ref{thm:light-fixed} hold with $\ell=\ell_1$. In addition, assume that for every $\ell\in(0,\ell_1)$ there exists $R_\ell>R_{\ell_1}$ such that
\[
\psi_{\rm cl}(u,\ell)\le e^{-R_\ell u},\qquad u\ge0.
\]
If $\psi$ is continuous or monotone on $(-\infty,0)$, then
\[
\psi(u)\sim C_{\ell_1}\,\psi_{\rm cl}(u),\qquad u\to\infty.
\]
\end{thm}

\begin{proof}
By Theorem~\ref{thm:light-fixed},
\[
\psi_{\ell_1}(u)\sim C_{\ell_1}\,\psi_{\rm cl}(u,\ell_1).
\]
Since
\[
\psi(u)=p_1\psi_{\ell_1}(u)+\int_{[0,\ell_1)}\psi_\ell(u)\,G(\rmd\ell)
\]
and $\psi_\ell(u)\le\psi_{\rm cl}(u,\ell)$, we obtain for all large $u$
\begin{align*}
&\frac{1}{\psi_{\rm cl}(u,\ell_1)}\int_{[0,\ell_1)}\psi_\ell(u)\,G(\rmd\ell)\\
&\le \frac{2}{c_{\ell_1}}\int_{[0,\ell_1)}e^{-(R_\ell-R_{\ell_1})u}\,G(\rmd\ell)\to0
\end{align*}
by dominated convergence. Hence
\[
\psi(u)\sim p_1C_{\ell_1}\,\psi_{\rm cl}(u,\ell_1).
\]
The same estimate with $\psi_\ell(u)$ replaced by $\psi_{\rm cl}(u,\ell)$ shows that
\[
\psi_{\rm cl}(u)\sim p_1\,\psi_{\rm cl}(u,\ell_1),
\]
and the result follows.
\end{proof}

\par\smallskip\noindent\textit{Remark.} Theorem~\ref{thm:light-mixed} covers the discrete endpoint case. The next theorem treats the complementary situation in which the endpoint is approached through a regular density and the leading term is produced by an $O(1/u)$ neighborhood of $\ell_1$.

\begin{thm}\label{thm:light-sharp}
Assume that $\ell_1<c/\mu$ and $G(\{\ell_1\})=0$. Suppose that there exist $\delta>0$, $B,C_1,D_1,R_1>0$, and $b>0$ such that:
\begin{enumerate}
\item$G$ has a density $g$ on $(\ell_1-\delta,\ell_1)$ and $g(\ell_1-z)\sim Bz^{b-1}$ as $z\downarrow0$;
\item the fixed-intensity theorem holds at $\ell=\ell_1$, with constant $C_{\ell_1}$, and for every $M>0$,
\[
\sup_{0\le v\le M}\left|\frac{\psi_{\rm cl}(u,\ell_1-v/(D_1u))}{C_1e^{-R_1u}e^{-v}}-1\right|\to0,
\]
\[
\sup_{0\le v\le M}\left|\frac{\psi_{\ell_1-v/(D_1u)}(u)}{\psi_{\rm cl}(u,\ell_1-v/(D_1u))}-C_{\ell_1}\right|\to0;
\]
\item the adjustment coefficient $R(\ell)$ exists on $(\ell_1-\delta,\ell_1]$, satisfies $R(\ell_1)=R_1$, is differentiable at $\ell_1$ with $R'(\ell_1)=-D_1$, and there is $\eta>0$ such that $R(\ell)\ge R_1+\eta$ for $\ell\le \ell_1-\delta$.
\end{enumerate}
If $\psi$ is continuous or monotone on $(-\infty,0)$, then
\[
\psi(u)\sim C_{\ell_1}\,\psi_{\rm cl}(u)\sim C_{\ell_1}\frac{BC_1\Gamma(b)}{(D_1u)^b}e^{-R_1u},\qquad u\to\infty.
\]
\end{thm}

\begin{proof}
Write $\psi(u)=I_0(u)+I_1(u)$, where $I_0$ and $I_1$ are the contributions of $[0,\ell_1-\delta]$ and $(\ell_1-\delta,\ell_1)$, respectively. Since $0\le \psi_\ell(u)\le \psi_{\rm cl}(u,\ell)$ and Lundberg's inequality gives $\psi_{\rm cl}(u,\ell)\le e^{-R(\ell)u}$,
\[
0\le I_0(u)\le e^{-(R_1+\eta)u},
\]
hence $u^be^{R_1u}I_0(u)\to0$. For $I_1$, the change of variables $v=D_1u(\ell_1-\ell)$ yields
\begin{align*}
u^be^{R_1u}I_1(u)
&=\frac{1}{D_1^b}\int_0^{D_1u\delta} e^{R_1u}\psi_{\ell_1-v/(D_1u)}(u)
\\
&\qquad\times \left(\frac{v}{D_1u}\right)^{1-b}g\left(\ell_1-\frac{v}{D_1u}\right)v^{b-1}\,\rmd v.
\end{align*}
For each fixed $v\ge0$, assumption~(2) gives
\begin{align*}
e^{R_1u}\psi_{\ell_1-v/(D_1u)}(u)
&=\frac{\psi_{\ell_1-v/(D_1u)}(u)}{\psi_{\rm cl}(u,\ell_1-v/(D_1u))}
\\
&\quad\times\frac{\psi_{\rm cl}(u,\ell_1-v/(D_1u))}{C_1e^{-R_1u}e^{-v}}
\\
&\quad\times C_1e^{-v}\to C_{\ell_1}C_1e^{-v}.
\end{align*}
while $\left(v/(D_1u)\right)^{1-b}g\left(\ell_1-v/(D_1u)\right)\to B$. Moreover, after shrinking $\delta$ if necessary, differentiability of $R$ at $\ell_1$ gives $R(\ell_1-z)\ge R_1+(D_1/2)z$ for $0<z<\delta$, while $g(\ell_1-z)\le 2Bz^{b-1}$. Since $\psi_\ell(u)\le \psi_{\rm cl}(u,\ell)\le e^{-R(\ell)u}$, the integrand is bounded by
\[
\frac{2B}{D_1^b}e^{-v/2}v^{b-1},
\]
which is integrable on $[0,\infty)$. Dominated convergence yields
\begin{align*}
u^be^{R_1u}I_1(u)
&\to C_{\ell_1}\frac{BC_1}{D_1^b}\int_0^\infty e^{-v}v^{b-1}\,\rmd v
\\
&= C_{\ell_1}\frac{BC_1\Gamma(b)}{D_1^b}.
\end{align*}
Hence
\[
\psi(u)\sim C_{\ell_1}\frac{BC_1\Gamma(b)}{(D_1u)^b}e^{-R_1u}.
\]
The same proof with $\psi_\ell(u)$ replaced by $\psi_{\rm cl}(u,\ell)$ gives
\[
\psi_{\rm cl}(u)\sim \frac{BC_1\Gamma(b)}{(D_1u)^b}e^{-R_1u},
\]
and the ratio statement follows.
\end{proof}

\par\smallskip\noindent\textit{Remark.} The assumptions of Theorem~\ref{thm:light-sharp} are local endpoint regularity conditions and are not intended to be minimal. Assumption~(1) covers beta-type endpoint densities; for instance, $g(\ell)=K(\ell_1-\ell)^{b-1}h(\ell)$ with $h(\ell_1-)>0$ gives $B=Kh(\ell_1-)$. Assumption~(2) is the local uniformity needed in the $O(1/u)$ window of the Laplace method, and is expected when the adjustment coefficient, Cram{\'e}r constant, and limiting overshoot law vary smoothly with the intensity, as in standard exponential or phase-type claim-size models. Assumption~(3) makes intensities bounded away from $\ell_1$ exponentially negligible. Thus the theorem covers regular densities near the maximal admissible intensity; more singular endpoints require separate analysis.

\subsection*{Acknowledgements}
This work was supported by JSPS KAKENHI Grant-in-Aid for Scientific Research (A), Grant Number JP25H00593, JSPS KAKENHI Grant-in-Aid for Scientific Research (C), Grant Numbers JP23K03133 and JP26K06828, and the JSPS Open Partnership Joint Research Projects, Grant Number JPJSBP120209921.
\references

\end{document}